\documentclass[12pt]{amsart}

\usepackage[T1]{fontenc}
\usepackage{amsmath,amssymb,amsthm,mathtools}
\usepackage{enumerate}
\usepackage[plainpages=false]{hyperref}

\hypersetup{
  colorlinks=true,
  linkcolor=blue,
  citecolor=blue,
  urlcolor=blue
}

\usepackage[
  a4paper,
  width=6.0in,
  height=8.25in,
  hmarginratio=1:1,
  vmarginratio=1:1
]{geometry}

\numberwithin{equation}{section}

\newtheorem{prop}{Proposition}[section]
\newtheorem{thm}[prop]{Theorem}
\newtheorem{lem}[prop]{Lemma}
\newtheorem{cor}[prop]{Corollary}
\theoremstyle{remark}
\newtheorem{rem}[prop]{Remark}

\DeclareMathOperator{\cof}{cof}
\DeclareMathOperator{\osc}{osc}
\DeclareMathOperator{\dist}{dist}
\DeclareMathOperator{\diam}{diam}
\DeclareMathOperator{\tr}{tr}
\newcommand{\R}{\mathbb R}
\newcommand{\Lg}{L_g}
\newcommand{\Lgs}{L_g^*}
\newcommand{\Div}{\operatorname{div}}

\allowdisplaybreaks
\title[Interior $C^{1,\alpha}$ estimates for the 2D linearized Monge--Amp\`ere equation]{Interior $C^{1,\alpha}$ estimates for the linearized Monge--Amp\`ere equation in two dimensions}

\author{Ling Wang}
\address{Department of Decision Sciences and BIDSA, Bocconi University, Milano, Italy}
\email{ling.wang@unibocconi.it}

\author{Bin Zhou}
\address{School of Mathematical Sciences, Peking University, Beijing 100871, China}
\email{bzhou@pku.edu.cn}


\begin{document}

\begin{abstract}
We prove an interior $C^{1,\alpha}$ estimate for solutions of the
homogeneous linearized Monge--Amp\`ere equation in dimension two under the
assumption
\[
0<\lambda\leq \det D^2\varphi\leq\Lambda<+\infty.
\]
No continuity assumption on the Monge--Amp\`ere density is required. 
Our result is an affine-invariant analogue of the classical
Morrey--Nirenberg $C^{1,\alpha}$ estimate in two dimensions. The core of the proof is the partial Legendre transform. After the transform, the first derivatives of
the solution are quotients of adjoint solutions for a uniformly elliptic
non-divergence form equation. Bauman's Harnack inequality gives the H\"older
control of the quotient, while the Jacobian identity of the partial Legendre
transform and a Caccioppoli estimate give its local boundedness. As an
application, we prove a Liouville theorem for entire solutions with at most
linear growth.
\end{abstract}
\maketitle

\section{Introduction}

This paper is concerned with the interior $C^{1,\alpha}$ estimate for solutions of the
homogeneous linearized Monge--Amp\`ere equation in dimension two. Let
$\Omega\subset\R^2$ be a bounded convex domain and $\varphi$ be a smooth and
strictly convex function with
\begin{equation}\label{eq:MA-basic}
0<\lambda\leq \det D^2\varphi\leq \Lambda<+\infty
\qquad\text{in }\Omega.
\end{equation}
The associated linearized Monge--Amp\`ere operator is
\[
L_\varphi u:=\Phi^{ij}D_{ij}u,
\qquad
\Phi:=\cof(D^2\varphi)
       =(\det D^2\varphi)(D^2\varphi)^{-1}.
\]
We use the Einstein summation convention throughout the paper. 
We mainly focus on the homogeneous equation
\begin{equation}\label{eq:Lphi-hom}
L_\varphi u=0.
\end{equation}

The difficulty in studying linearized Monge--Amp\`ere equations is that \eqref{eq:MA-basic} usually does not give a uniform ellipticity in
fixed Euclidean balls. The natural geometry is instead given by the
sections of $\varphi$,
\[
S_\varphi(x_0,h)
:=\bigl\{x\in\Omega:\
\varphi(x)<\varphi(x_0)+\nabla\varphi(x_0)\cdot(x-x_0)+h
\bigr\}.
\]
Caffarelli and Guti\'errez \cite{CG} proved the Harnack inequality and interior
H\"older estimate for solutions of the homogeneous linearized Monge--Amp\`ere
equation under a $\mathcal{A}_\infty$ condition. In particular, this condition is guaranteed if we have \eqref{eq:MA-basic}.
Subsequent developments include interior $W^{2,p}$ and second derivative
estimates, boundary regularity, global $W^{1,p}$ and $W^{2,p}$ estimates, and
applications to semigeostrophic equations. See
\cite{GN2,GuTo,LN1,LN2,LS,LeSem} and the references therein.
Linearized Monge--Amp\`ere operators also arise in affine maximal surface and related fourth-order equations. See, for example, \cite{TrWa,TW}. Linearized
Monge--Amp\`ere equations with drifts, together with applications to singular
Abreu equations, were studied in \cite{KLWZ}. For the geometry of
sections and the interior regularity theory of the Monge--Amp\`ere equation,
see \cite{CaLoc,Ca,CaGuRA,GH,GuBook,FiBook,LeBook}.

For gradient estimates, the known results require additional information on
the Monge--Amp\`ere density. Guti\'errez and Nguyen \cite{GN1} proved
interior H\"older estimates for the gradient when $\det D^2\varphi$ is
continuous. Tang and Zhang \cite{TZ} obtained related $C^{1,\alpha}$
estimates under a VMO condition on the density. More recently,
Cui \cite{Cui26} established gradient potential estimates and derived
criteria for continuity and local $C^{1,\gamma}$ regularity of the gradient.
Thus, the known gradient estimates require assumptions on the
Monge--Amp\`ere density beyond the two-sided bound \eqref{eq:MA-basic}.

The situation in dimension two is different. The classical results of
Morrey \cite{Mor} and Nirenberg \cite{Ni} give interior $C^{1,\alpha}$
estimates for uniformly elliptic equations in two variables without the continuity assumption on the coefficients. The same first-order estimate does not hold, in general, in higher dimensions \cite{Saf}. Although $L_\varphi$ is not uniformly
elliptic in Euclidean coordinates, it is affine invariant and the sections of
$\varphi$ provide the natural replacement for Euclidean balls. We therefore ask whether an affine-invariant version of the planar
Morrey--Nirenberg estimate holds for the linearized Monge--Amp\`ere equation.

In the positive direction, Savin \cite{Sa} provided further evidence. In particular, he proved that every entire globally Lipschitz solution of
\eqref{eq:Lphi-hom} in $\R^2$ under the assumption \eqref{eq:MA-basic} is affine. He also observed that this Liouville
theorem suggests an interior $C^{1,\alpha}$ estimate in dimension two and
indicated that this problem would be addressed in a subsequent work. To the best of our knowledge, however, this estimate has not yet appeared in the literature. In the present paper, we prove such an estimate assuming only \eqref{eq:MA-basic}.

After subtracting the tangent plane at the center of a section and applying the
standard affine normalization, we assume
$\varphi(0)=0$,
$\nabla\varphi(0)=0$,
and write $S_\theta:=S_\varphi(0,\theta)$ for $0<\theta\leq 1$. We use the standard
notion of a normalized section.

\begin{thm}[Interior $C^{1,\alpha}$ estimate]\label{thm:main}
Let $S_1\Subset\Omega$ be a normalized section. Assume that $\varphi$ is
smooth and strictly convex in a neighborhood of $\overline{S_1}$ and
\[
0<\lambda\leq \det D^2\varphi\leq\Lambda<+\infty
\qquad\text{in }S_1.
\]
Suppose $u\in C^\infty(S_1)\cap L^\infty(S_1)$ is a solution to \eqref{eq:Lphi-hom} in $S_1$. Then for every $\theta\in(0,1)$, there exist $\alpha\in(0,1)$ and
$C>0$, depending only on $\lambda,\Lambda$ and $\theta$, such that
\begin{equation}\label{eq:main-estimate-intro}
\|\nabla u\|_{C^\alpha(S_\theta)}
\leq C\,\osc_{S_1}u.
\end{equation}
\end{thm}

Theorem~\ref{thm:main} may be viewed as an affine-invariant analogue of the
Morrey--Nirenberg estimate. The formulation on normalized sections is
essential. The determinant bound \eqref{eq:MA-basic} does not control the
Euclidean ellipticity ratio of $L_\varphi$, so an estimate with constants
depending only on $\lambda$ and $\Lambda$ cannot in general be stated on fixed
Euclidean balls.

\vskip 8pt

The proof of Theorem \ref{thm:main} is based on the partial Legendre transform. This transform is
classical in the study of Monge--Amp\`ere equations. 
To the best of our knowledge, it was first used for the linearized
Monge--Amp\`ere equation by the first author in \cite{Wa}, in which a new proof of the Caffarelli--Guti\'errez H\"older estimate in dimension two is given. 
Here, we apply the transform to the first derivatives of $u$. Write
\[
\varphi^\star(\xi,\eta)
:=\sup_{x_1}\{x_1\xi-\varphi(x_1,\eta)\},
\qquad
(\xi,\eta)=\mathcal P(x_1,x_2)
:=(\varphi_1(x_1,x_2),x_2),
\]
and let $\mathcal Q=\mathcal P^{-1}$. Set
\[
a:=\varphi^\star_{\xi\xi},
\qquad
g:=\det D^2\varphi\circ \mathcal Q,
\qquad
\widetilde u:=u\circ \mathcal Q.
\]
Then
\begin{equation}\label{eq:transformed-intro}
\varphi^\star_{\eta\eta}+g\varphi^\star_{\xi\xi}=0,
\qquad
(g\widetilde u_\xi)_\xi+\widetilde u_{\eta\eta}=0,
\qquad
\lambda\leq g\leq\Lambda.
\end{equation}
Denote
\[
\Lg:=g\partial_{\xi\xi}+\partial_{\eta\eta},
\qquad
\Lgs m:=\partial_{\xi\xi}(gm)+\partial_{\eta\eta}m.
\]
It is easy to see that $\Lg$ is uniformly elliptic 
and $\Lgs$ is its adjoint operator. If
\[
p:=u_1\circ \mathcal Q=\frac{\widetilde u_\xi}{a},
\]
then differentiating the two equations in \eqref{eq:transformed-intro} gives
\[
L_g^*a=0,
\qquad
L_g^*(ap)=0.
\]
Hence $p$ is a quotient of two solutions of the same adjoint equation.

Bauman's Harnack inequality for normalized adjoint solutions \cite{Ba} gives the
H\"older control of this quotient once $p$ is locally bounded. It remains to obtain a local bound for $p$. The Jacobian formula
\[
a(\xi,\eta)\,\mathrm d\xi\mathrm d\eta=\mathrm dx
\]
and the interior $C^{1,\beta}$ estimate for $\varphi$ give a lower bound for
the mass of $a$ on transformed balls. Since
$\widetilde u_\xi=ap$, a Caccioppoli estimate for $\widetilde u$ then gives a
local bound for $p$. Bauman's Harnack inequality applied to the adjoint
quotient yields the H\"older estimate for $p$. The estimate for $u_1$ follows
by composing with $\mathcal P$, and the argument for $u_2$ is the same after
interchanging the two variables.

\vskip 8pt

As an application, we obtain a Liouville theorem in which the global Lipschitz
assumption in \cite[Theorem~1.1]{Sa} is replaced by at most linear growth.

\begin{thm}[Liouville theorem]\label{thm:Liouville}
Let $\varphi:\R^2\to\R$ be smooth and strictly convex, and assume
\[
0<\lambda\leq \det D^2\varphi\leq\Lambda<+\infty
\qquad\text{in }\R^2.
\]
Let $u\in C^\infty(\R^2)$ be a solution to \eqref{eq:Lphi-hom}
in $\R^2$. If
\[
|u(x)|\leq C_1(1+|x|)
\qquad\text{for all }x\in\R^2,
\]
then $u$ is an affine function.
\end{thm}

To see the role of Theorem~\ref{thm:main}, normalize the large sections
$S_h=S_\varphi(0,h)$ by affine maps with linear parts $A_h$. The normalized
estimate controls $A_h^{-T}\nabla u$ on a fixed section. Instead of
estimating the full gradient, we project the gradient difference onto
a right singular vector $e_h$ associated with the smallest singular value of $A_h$. 
The resulting factor
$\|A_h^{-1}\|^{-1}$ compensates for the linear growth of $u$ along the long
axis of $S_h$. Passing to a sequence $h\to\infty$ gives a fixed direction $e$
with $D^2u\,e=0$ in $\R^2$. Since the dimension is two, the equation then
implies $D^2u=0$.

We work throughout with smooth potentials and smooth solutions, so all
changes of variables and differentiations are classical. The estimates are
\textit{a priori}. In particular, the constants do not depend on higher
derivatives of $\varphi$ or $u$.

The rest of the paper is organized as follows. In Section~\ref{sec:prelim}, we
recall the properties of normalized sections used in the proof. In
Section~\ref{sec:PL-quotient}, we derive the equations under the partial
Legendre transform and the estimate for quotients of adjoint solutions. In
Section~\ref{sec:bound-main}, we prove the local boundedness of the quotient
and Theorem~\ref{thm:main}. Theorem~\ref{thm:Liouville} is proved in
Section~\ref{sec:liouville}.

\medskip
\noindent\textbf{Acknowledgments.}
The authors would like to thank Prof.~Ovidiu Savin for his work \cite{Sa},
which motivated the problem considered in this paper. The first author is also
grateful to his postdoctoral mentor, Prof.~Antonio De Rosa, for generous
support and encouragement.

The first author was funded by the European Union through the European Research
Council (ERC), under the Starting Grant ``ANGEVA'' (grant agreement No.
101076411). Views and opinions expressed are, however, those of the authors
only and do not necessarily reflect those of the European Union or the
European Research Council. Neither the European Union nor the granting
authority can be held responsible for them.  The second author 
is partially supported by  National Key R$\&$D Program of China 2023YFA009900 and NSFC  Grant 12271008.

\medskip
\noindent\textbf{AI Disclosure.} 
During the preparation of this manuscript, the authors used ChatGPT-5.6 Sol Plus for language polishing and to improve the readability of drafts written by the authors, particularly the abstract and introduction. The tool was also used in discussions concerning the organization and presentation of parts of the proof of Theorem \ref{thm:main}, including the presentation of the adjoint quotient arising after the partial Legendre transform, and to suggest potentially relevant references. The mathematical problem, the use of the partial Legendre transform, and the overall proof strategy were developed by the authors. All mathematical arguments, results, and cited references were independently checked by the authors, who take full responsibility for the content of the paper.

\section{Preliminaries}\label{sec:prelim}

Throughout the paper, $C$ and $c$ denote positive constants which may change
from line to line. Their dependence will be indicated when necessary. For a
bounded set $E$, we write
\[
\osc_E u:=\sup_E u-\inf_E u,
\qquad
\|f\|_{C^\alpha(E)}:=\|f\|_{L^\infty(E)}+[f]_{C^\alpha(E)}.
\]
In Sections~\ref{sec:prelim}--\ref{sec:bound-main}, $S_1=S_\varphi(0,1)$ is
a normalized section as in the introduction and
$S_\theta=S_\varphi(0,\theta)$. Unless otherwise stated, we work under the
hypotheses of Theorem~\ref{thm:main}.

We first recall the interior $C^{1,\beta}$ estimate and strict convexity of
solutions to the Monge--Amp\`ere equation. The strict convexity statement is
specific to dimension two.

\begin{lem}\label{lem:Caffarelli-C1beta}
Let $S_1$ be normalized and let $\varphi$ be convex in $S_1$ with
\[
\lambda\leq \det D^2\varphi\leq\Lambda
\]
in the Alexandrov sense. For every $\theta\in(0,1)$ there exist
$\beta\in(0,1)$ and $C<\infty$, depending only on
$\lambda,\Lambda$ and $\theta$, such that
\[
\|\nabla\varphi\|_{C^\beta(S_\theta)}\leq C.
\]
Moreover, $\overline{S_\theta}\Subset S_1$ quantitatively, and $\varphi$ is
strictly convex in the interior of $S_1$.
\end{lem}

\begin{proof}
The $C^{1,\beta}$ estimate and the quantitative compact inclusion of lower
sections follow from Caffarelli's interior theory. See
\cite{Ca,CaLoc,GuBook,FiBook,LeBook}. In dimension two, interior strict
convexity is classical. See \cite{He,CaLoc}.
\end{proof}

For the partial Legendre transform in the $x_1$-variable, let
\[
\mathcal P(x_1,x_2):=(\varphi_1(x_1,x_2),x_2).
\]
Since $\varphi_{11}>0$, the map $\mathcal P$ is locally a diffeomorphism. For
fixed $x_2$, the function $x_1\mapsto\varphi_1(x_1,x_2)$ is strictly
increasing. It follows that $\mathcal P$ is one-to-one on a section and hence
is a diffeomorphism onto its image.

\begin{lem}\label{lem:P-interior-control}
Let $0<\theta<\theta_1<1$. Denote $Q_h:=\mathcal P(S_h)$ for $h\leq 1$.
There exist $d_0>0$ and $D_0<\infty$, depending only on
$\lambda,\Lambda,\theta$ and $\theta_1$, such that
\begin{equation}\label{eq:P-distance-diam}
\dist(Q_\theta,\partial Q_{\theta_1})\geq d_0,
\qquad
\diam Q_\theta\leq D_0.
\end{equation}
The same conclusion holds for the map
$(x_1,x_2)\mapsto(x_1,\varphi_2(x_1,x_2))$.
\end{lem}

\begin{proof}
The diameter bound follows from Lemma~\ref{lem:Caffarelli-C1beta}. The first
component of $\mathcal P$ is $\varphi_1$, and the second component is $x_2$.
Both are uniformly bounded on a fixed lower section.

We prove the separation estimate. Since $\overline{S_{\theta_1}}\Subset S_1$,
the restriction of $\mathcal P$ to $\overline{S_{\theta_1}}$ is a
homeomorphism onto its image. Hence
\[
\overline{Q_{\theta_1}}=\mathcal P(\overline{S_{\theta_1}}),
\qquad
\partial Q_{\theta_1}=\mathcal P(\partial S_{\theta_1}).
\]
Suppose by contradiction that the separation estimate fails. Then there are
normalized potentials $\varphi_k$ satisfying the same determinant bounds and
points $x^{(k)}\in S_\theta^{(k)}$, $y^{(k)}\in\partial S_{\theta_1}^{(k)}$,
such that
\begin{equation}\label{eq:P-separation-contradiction}
|\mathcal P_k(x^{(k)})-\mathcal P_k(y^{(k)})|\to0,
\qquad
\mathcal P_k=(\partial_1\varphi_k,x_2).
\end{equation}
Set $\theta_2=\frac{1+\theta_1}{2}$. By normalized compactness, after passing to a
subsequence the potentials converge locally uniformly to a normalized convex
Alexandrov solution $\varphi_\infty$ with the same determinant bounds. The
quantitative inclusion of lower sections and Lemma~\ref{lem:Caffarelli-C1beta}
give $C^1$ convergence on a neighborhood of the points under consideration.
We may also assume
$x^{(k)}\to x_\infty$, $y^{(k)}\to y_\infty$.
The convergence of the sections gives
\[
\varphi_\infty(x_\infty)\leq\theta,
\qquad
\varphi_\infty(y_\infty)=\theta_1,
\]
so $x_\infty\neq y_\infty$. On the other hand,
\eqref{eq:P-separation-contradiction} and the $C^1$ convergence imply
\[
(x_\infty)_2=(y_\infty)_2,
\qquad
(\varphi_\infty)_1(x_\infty)
=(\varphi_\infty)_1(y_\infty).
\]
The horizontal segment joining $x_\infty$ and $y_\infty$ is contained in
$\{\varphi_\infty\leq\theta_1\}\Subset S_1$. The restriction of
$\varphi_\infty$ to
this segment is a one-dimensional convex $C^1$ function whose derivative has
the same value at the two endpoints. It is therefore affine on the segment,
contrary to Lemma~\ref{lem:Caffarelli-C1beta}.

The proof for $(x_1,x_2)\mapsto(x_1,\varphi_2)$ is identical.
\end{proof}

\begin{rem}
The separation in Lemma~\ref{lem:P-interior-control} can also be obtained from
a quantitative modulus of convexity. Recall that, for a differentiable convex
function $\varphi$, one may define the local modulus of convexity
\[
m_\varphi(t;K):=
\inf\Bigl\{
\varphi(y)-\varphi(x)-\nabla\varphi(x)\cdot(y-x):
x,y\in K,\ |x-y|\geq t
\Bigr\},
\]
where $K\Subset\Omega$. In dimension two, the Aleksandrov--Heinz theorem gives
strict convexity for generalized solutions whose Monge--Amp\`ere measure has a
positive lower bound. See Heinz \cite{He} and also
\cite[Remark~3.2]{TW} for an elementary proof. In a normalized compact family
of sections, this strict convexity is quantitative on compact subsets. More
precisely, if $K$ is quantitatively contained in $S_1$, then for every $t>0$
there is a number $\omega(t)>0$,
depending only on $\lambda,\Lambda,t$ and the relative position of $K$ in
$S_1$, such that
\begin{equation}\label{eq:modulus-lower}
m_\varphi(t;K)\geq \omega(t).
\end{equation}
Indeed, otherwise normalized compactness would give a sequence of potentials
and pairs of points, staying in a fixed lower section and a fixed positive
distance apart, for which the left-hand side tends to zero. After passing to a
subsequence, the potentials converge in $C^1$ on a lower section. The limiting
potential would then be affine on a nontrivial segment, contradicting planar
strict convexity. For a quantitative formula for the modulus in dimension two,
see \cite[Lemma~2.5]{Liu}. Its proof develops the argument in
\cite[Section~3]{TW} and gives a lower bound in terms of a boundary gradient
quantity.

Let us indicate how such a modulus gives another proof of the positive
separation in Lemma~\ref{lem:P-interior-control}. Fix
$0<\theta<\theta_1<1$ and set $\theta_2=\frac{1+\theta_1}{2}$.
By the compact inclusion of lower sections and
Lemma~\ref{lem:Caffarelli-C1beta}, there are $C_0<\infty$, $s_0>0$ and
$\delta_0>0$ such that
\[
\{z:\dist(z,\overline{S_{\theta_1}})\leq\delta_0\}
\Subset S_{\theta_2},
\qquad
\|\nabla\varphi\|_{C^\beta(S_{\theta_2})}\leq C_0,
\]
and
\begin{equation}\label{eq:section-point-separation}
|x-y|\geq s_0
\end{equation}
whenever $x\in S_\theta$ and $y\in\partial S_{\theta_1}$. We use
\eqref{eq:modulus-lower} below with $K=\overline{S_{\theta_2}}$. To see the
last inequality, note that the segment joining $x$ and $y$ is contained in
$\overline{S_{\theta_1}}$, and hence
\[
\theta_1-\theta
\leq\varphi(y)-\varphi(x)
\leq C_0|x-y|.
\]

Suppose now that $|\mathcal P(x)-\mathcal P(y)|$ is small for some
$x\in S_\theta$ and $y\in\partial S_{\theta_1}$. Then both
$|x_2-y_2|$ and $|\varphi_1(x)-\varphi_1(y)|$
are small. Let $\bar y:=(y_1,x_2)$.
If $|x_2-y_2|$ is sufficiently small, then
$|\bar y-y|<\delta_0$. Hence $\bar y\in S_{\theta_2}$, and the vertical
segment joining $y$ and $\bar y$ also lies in $S_{\theta_2}$. The
$C^{1,\beta}$ estimate gives
\begin{equation}\label{eq:vertical-gradient-control}
|\varphi_1(\bar y)-\varphi_1(y)|
\leq C_0|x_2-y_2|^\beta.
\end{equation}
It follows that $|\varphi_1(\bar y)-\varphi_1(x)|$ is small. On the other
hand, \eqref{eq:section-point-separation} and the smallness of $|x_2-y_2|$
imply $|x_1-y_1|\geq\frac{s_0}{2}$.
By convexity, the horizontal segment joining $x$ and $\bar y$ lies in
$S_{\theta_2}$. Therefore \eqref{eq:modulus-lower} gives
\begin{equation}\label{eq:horizontal-modulus}
\varphi(\bar y)-\varphi(x)
-\nabla\varphi(x)\cdot(\bar y-x)
\geq\omega(\frac{s_0}{2})>0.
\end{equation}
For a differentiable convex function $f$ on an interval,
\[
0\leq f(b)-f(a)-f'(a)(b-a)
\leq |b-a|\,|f'(b)-f'(a)|.
\]
Applying this to $f(t)=\varphi(t,x_2)$ and using
$|x_1-y_1|\leq\diam S_{\theta_2}\leq C$, we obtain from
\eqref{eq:horizontal-modulus}
\[
|\varphi_1(\bar y)-\varphi_1(x)|\geq c>0,
\]
which contradicts \eqref{eq:vertical-gradient-control} and the assumed
smallness of $|\mathcal P(x)-\mathcal P(y)|$. This gives another proof of the
first estimate in \eqref{eq:P-distance-diam}. 
\end{rem}

\section{Partial Legendre transform and adjoint quotients}\label{sec:PL-quotient}

We first derive the equation under the partial Legendre transform. Let
\[
(\xi,\eta)=\mathcal P(x_1,x_2):=(\varphi_1(x_1,x_2),x_2),
\]
and write $\mathcal Q=\mathcal P^{-1}$. The partial Legendre transform of $\varphi$ in
the $x_1$-variable is
\[
\varphi^\star(\xi,\eta)
=\sup_{x_1}\{x_1\xi-\varphi(x_1,\eta)\}
=x_1\xi-\varphi(x_1,\eta),
\]
where $(\xi,\eta)=\mathcal P(x_1,\eta)$. We have
\[
\varphi^\star_\xi=x_1,
\qquad
\varphi^\star_\eta=-\varphi_2(\mathcal Q(\xi,\eta)).
\]
Differentiating once more gives
\begin{equation}\label{eq:phi-star-second}
\varphi^\star_{\xi\xi}=\frac1{\varphi_{11}},
\qquad
\varphi^\star_{\xi\eta}=-\frac{\varphi_{12}}{\varphi_{11}},
\qquad
\varphi^\star_{\eta\eta}=-\frac{\det D^2\varphi}{\varphi_{11}},
\end{equation}
where the derivatives of $\varphi$ on the right-hand side are evaluated at
$\mathcal Q(\xi,\eta)$. Set
\[
a:=\varphi^\star_{\xi\xi}>0,
\qquad
c:=\varphi^\star_{\xi\eta},
\qquad
g:=\det D^2\varphi\circ \mathcal Q.
\]
Then
\begin{equation}\label{eq:phi-star-linear}
\varphi^\star_{\eta\eta}+g\varphi^\star_{\xi\xi}=0,
\qquad
\lambda\leq g\leq\Lambda.
\end{equation}
We write
\[
\Lg:=g\partial_{\xi\xi}+\partial_{\eta\eta},
\qquad
\Lgs m:=\partial_{\xi\xi}(gm)+\partial_{\eta\eta}m.
\]

\begin{lem}
Let $u$ be a solution to $L_\varphi u=0$ on a section. Denote
$p:=u_1\circ \mathcal Q$.
With the notation above, we have
\begin{equation}\label{eq:adjoint-a-ap}
\Lgs a=0,
\qquad
\Lgs(ap)=0.
\end{equation}
\end{lem}

\begin{proof}
Let $A:=\varphi^\star_\xi=x_1$. Differentiating
\eqref{eq:phi-star-linear} in $\xi$ gives
\[
(gA_\xi)_\xi+A_{\eta\eta}=0.
\]
Since $A_\xi=a$, another differentiation in $\xi$ gives
\[
\partial_{\xi\xi}(ga)+\partial_{\eta\eta}a=0.
\]
This is the first equation in \eqref{eq:adjoint-a-ap}.

Since $\Phi$ is divergence free
and sections are simply connected, there is a function $v$, unique up
to an additive constant, such that
\[
\nabla v=J\Phi\nabla u,
\qquad
J=\begin{pmatrix}0&-1\\1&0\end{pmatrix}.
\]
Thus
\begin{equation}\label{eq:stream-components}
v_1=\varphi_{12}u_1-\varphi_{11}u_2,
\qquad
v_2=\varphi_{22}u_1-\varphi_{12}u_2.
\end{equation}
Define
\[
\widetilde u:=u\circ \mathcal Q,
\qquad
\widetilde v:=v\circ \mathcal Q.
\]
Since
\[
x_{1,\xi}=\frac1{\varphi_{11}},
\qquad
x_{1,\eta}=-\frac{\varphi_{12}}{\varphi_{11}},
\]
we have
\begin{equation}\label{eq:utilde-derivatives}
\widetilde u_\xi=\frac{u_1}{\varphi_{11}},
\qquad
\widetilde u_\eta
=u_2-\frac{\varphi_{12}}{\varphi_{11}}u_1.
\end{equation}
Combining \eqref{eq:stream-components} and
\eqref{eq:utilde-derivatives}, we obtain
\[
\widetilde v_\xi=-\widetilde u_\eta,
\qquad
\widetilde v_\eta=g\widetilde u_\xi.
\]
Hence
\begin{equation}\label{eq:utilde-div}
(g\widetilde u_\xi)_\xi+\widetilde u_{\eta\eta}=0.
\end{equation}
By \eqref{eq:phi-star-second} and \eqref{eq:utilde-derivatives},
$\widetilde u_\xi=ap$. Differentiating \eqref{eq:utilde-div} in $\xi$ gives
\[
\partial_{\xi\xi}(gap)+\partial_{\eta\eta}(ap)=0,
\]
which is the second equation.
\end{proof}

\begin{rem}
There is also a divergence form equation for $p$. If
$F=(\varphi^\star_\xi,-\varphi^\star_\eta)$, then
\[
DF=
\begin{pmatrix}
a&c\\
-c&ga
\end{pmatrix}
\]
and \eqref{eq:adjoint-a-ap} gives
\begin{equation}\label{eq:p-div-cof}
\Div\left(\cof DF\,\nabla p\right)=0,
\end{equation}
where 
\[\cof DF=
\begin{pmatrix}
ga&c\\
-c&a
\end{pmatrix}\] 
is the cofactor of $DF$. Indeed, using
$c_\eta=-(ga)_\xi$ and $c_\xi=a_\eta$, the left-hand side equals
\[
ga\,p_{\xi\xi}+a\,p_{\eta\eta}
+2(ga)_\xi p_\xi+2a_\eta p_\eta,
\]
which is the equation obtained from
$L_g^*(ap)-pL_g^*a=0$.
However, equation \eqref{eq:p-div-cof} is not uniformly elliptic, since the symmetric
part of $\cof DF$ is
\[
\begin{pmatrix}ga&0\\0&a\end{pmatrix}
\]
and $a$ has no uniform upper or lower bound. Dividing by $a$ gives
\[
gp_{\xi\xi}+p_{\eta\eta}
+2\frac{(ga)_\xi}{a}p_\xi
+2\frac{a_\eta}{a}p_\eta=0.
\]
Here
$\frac{(ga)_\xi}{a}=g_\xi+g\frac{a_\xi}{a}$.
So the lower order terms involve derivatives of $g$. No estimate for these
terms follows from \eqref{eq:MA-basic}. We therefore keep the adjoint system
\eqref{eq:adjoint-a-ap}, which avoids differentiating $g$.
\end{rem}
\vskip 10pt

We recall Bauman's Harnack inequality that will be used below.
Let $D\subset\R^n$ be a smooth bounded domain, and let
$\mathcal L=b^{ij}D_{ij}$ with smooth symmetric coefficients satisfying
\[
\lambda_0|\zeta|^2
\leq
b^{ij}(z)\zeta_i\zeta_j
\leq
\Lambda_0|\zeta|^2
\qquad
\text{for all }z\in D,\ \zeta\in\R^n.
\]
Fix $x_0\in D$, and let $G_D(x_0,\cdot)$ denote the positive Green function
of $\mathcal L$ in $D$ with pole at $x_0$. In the second variable,
\[
\mathcal L^*G_D(x_0,\cdot)=0
\qquad
\text{in }D\setminus\{x_0\},
\]
where
\[
\mathcal L^*v:=D_{ij}(b^{ij}v).
\]
Following Bauman \cite{Ba}, a function $q$ on an open set
$U\Subset D\setminus\{x_0\}$ is called a \emph{normalized adjoint solution}
relative to $G_D(x_0,\cdot)$ if
\[
\mathcal L^*\bigl(qG_D(x_0,\cdot)\bigr)=0
\qquad
\text{in }U.
\]
Since $G_D(x_0,\cdot)$ itself solves the adjoint equation away from its pole,
the normalized adjoint solutions form a vector space containing the constants.
In particular, if $q$ is a normalized adjoint solution and $c\in\R$,
then $q-c$ is again a normalized adjoint solution on the same set. This
invariance allows one to pass directly from Bauman's Harnack inequality to a
quantitative oscillation estimate.

\begin{thm}
\label{prop:Bauman-normalized}
Let
$B_{2r}(z_0)\Subset D\setminus\{x_0\}$
and let $q$ be a smooth normalized adjoint solution in
$B_{2r}(z_0)$.
If $q\geq0$ in $B_{2r}(z_0)$, then
\begin{equation}\label{eq:Bauman-Harnack}
\sup_{B_r(z_0)}q
\leq
H\inf_{B_r(z_0)}q,
\end{equation}
where $H\geq1$ depends only on
$n,\lambda_0$, and $\Lambda_0$.

For a normalized adjoint solution of arbitrary sign, there exist
$\gamma\in(0,1)$ and $C<\infty$, depending only on the same data, such
that
\begin{equation}\label{eq:Bauman-normalized-holder}
r^\gamma[q]_{C^\gamma(B_r(z_0))}
\leq
C\,\osc_{B_{2r}(z_0)}q.
\end{equation}
The constants depend only on $n,\lambda_0$ and $\Lambda_0$.
\end{thm}

\begin{proof}
Bauman's Harnack principle \cite[Theorem~4.4]{Ba} is stated with
$B_{4r}(z_0)$ in place of $B_{2r}(z_0)$. The form
\eqref{eq:Bauman-Harnack} follows from that result by a finite Harnack chain
inside $B_{2r}(z_0)$. We prove the H\"older estimate from this form of the
Harnack inequality. See also \cite[Theorem~4.5]{Ba}.

Let
$B_{2s}(z)\Subset B_{2r}(z_0)
$
and introduce the notation
\[
m:=\inf_{B_{2s}(z)}q,
\qquad
M:=\sup_{B_{2s}(z)}q,
\]
and
\[
m_1:=\inf_{B_s(z)}q,
\qquad
M_1:=\sup_{B_s(z)}q.
\]
Both $q-m$ and $M-q$
are normalized adjoint solutions in $B_{2s}(z)$. They are nonnegative
there, so \eqref{eq:Bauman-Harnack} gives
\begin{align}
M_1-m
&\leq
H(m_1-m),                                      \label{eq:Bauman-lower-gap}\\
M-m_1
&\leq
H(M-M_1).                                      \label{eq:Bauman-upper-gap}
\end{align}
Set
\[
\omega:=M-m=\osc_{B_{2s}(z)}q,
\qquad
\omega_1:=M_1-m_1=\osc_{B_s(z)}q.
\]
Since
\[
M_1-m=(m_1-m)+\omega_1,
\]
inequality \eqref{eq:Bauman-lower-gap} implies
\[
\omega_1\leq(H-1)(m_1-m).
\]
Similarly, \eqref{eq:Bauman-upper-gap} implies
\[
\omega_1\leq(H-1)(M-M_1).
\]
If $H=1$, then $\omega_1=0$. If $H>1$, the preceding inequalities
give
\[
m_1-m\geq\frac{\omega_1}{H-1},
\qquad
M-M_1\geq\frac{\omega_1}{H-1}.
\]
Using
\[
\omega
=
(m_1-m)+\omega_1+(M-M_1),
\]
we conclude in both cases that
\begin{equation}\label{eq:Bauman-oscillation-decay}
\osc_{B_s(z)}q
\leq
\vartheta\,\osc_{B_{2s}(z)}q,
\end{equation}
where $\vartheta:=\frac{H-1}{H+1}\in[0,1)$.

Fix $x\in B_r(z_0)$. Since
$\overline{B_r(x)}\subset B_{2r}(z_0)$,
we may apply \eqref{eq:Bauman-oscillation-decay} successively to the
concentric balls
$B_r(x),\ B_{r/2}(x),\ B_{r/4}(x),\cdots$.
Thus, for every integer $k\geq0$,
\begin{equation}\label{eq:Bauman-dyadic-decay}
\osc_{B_{2^{-k}r}(x)}q
\leq
\vartheta^k\osc_{B_r(x)}q
\leq
\vartheta^k\osc_{B_{2r}(z_0)}q.
\end{equation}
Choose $\gamma\in(0,1)$, depending only on $H$, such that
$\vartheta\leq2^{-\gamma}$.
For instance, if $\vartheta>0$, one may take
\[
\gamma
=
\min\left\{
\frac12,
-\frac{\log\vartheta}{\log2}
\right\},
\]
while if $\vartheta=0$, one may take $\gamma=\frac12$.

Given $0<\rho\leq r$, choose $k\geq0$ such that
\[
2^{-(k+1)}r<\rho\leq2^{-k}r.
\]
It follows from \eqref{eq:Bauman-dyadic-decay} that
\begin{equation}\label{eq:Bauman-power-decay}
\osc_{B_\rho(x)}q
\leq
C\left(\frac{\rho}{r}\right)^\gamma
\osc_{B_{2r}(z_0)}q.
\end{equation}

Let $x,y\in B_r(z_0)$. If $0<|x-y|<r/2$, then
$y\in B_{2|x-y|}(x)$,
and \eqref{eq:Bauman-power-decay} yields
\[
|q(x)-q(y)|
\leq
C\left(\frac{|x-y|}{r}\right)^\gamma
\osc_{B_{2r}(z_0)}q.
\]
If $|x-y|\geq r/2$, the same estimate follows from
\[
|q(x)-q(y)|
\leq
\osc_{B_{2r}(z_0)}q
\leq
2^\gamma
\left(\frac{|x-y|}{r}\right)^\gamma
\osc_{B_{2r}(z_0)}q.
\]
Taking the supremum over $x,y\in B_r(z_0)$ proves
\eqref{eq:Bauman-normalized-holder}.
\end{proof}

We shall use the following consequence of Theorem~\ref{prop:Bauman-normalized}.
The numerator and denominator are normalized by the same Green function, which
cancels in the quotient.

\begin{cor}
\label{cor:Bauman-quotient}
Let $g\in C^\infty(B_{8r}(z_0))$ satisfy
\[
\lambda\leq g\leq\Lambda
\qquad
\text{in }B_{8r}(z_0).
\]
Let $a>0$ and $w$ be smooth functions satisfying
\[
L_g^*a=0,
\qquad
L_g^*w=0
\qquad
\text{in }B_{8r}(z_0).
\]
Suppose that
$|w|\leq Ma$ in $B_{2r}(z_0)$ for some $M\geq0$. Then there exist $\gamma\in(0,1)$ and
$C<\infty$, depending only on $\lambda$ and $\Lambda$, such that
\begin{equation}\label{eq:Bauman-quotient-estimate}
\left\|\frac{w}{a}\right\|_{L^\infty(B_r(z_0))}
+
r^\gamma
\left[\frac{w}{a}\right]_{C^\gamma(B_r(z_0))}
\leq
CM.
\end{equation}
In particular, if $w=ap$, then
\[
\|p\|_{L^\infty(B_r(z_0))}
+
r^\gamma[p]_{C^\gamma(B_r(z_0))}
\leq
C\|p\|_{L^\infty(B_{2r}(z_0))}.
\]
No modulus of continuity or higher norm of $g$ enters the estimate.
\end{cor}

\begin{proof}
By translation and scaling, it is enough to consider $z_0=0$ and $r=1$.
The case $M=0$ is immediate.

Choose $x_0\in B_8\setminus\overline{B_6}$, and let
$G=G_{B_8}(x_0,\cdot)$ be the Green function used in the definition of
normalized adjoint solutions. Set
\[
q_0:=\frac{a}{G}
\qquad\text{in }B_4,
\]
and
\[
q_\pm:=\frac{Ma\pm w}{G}
\qquad\text{in }B_2.
\]
Then $q_0$ is a positive normalized adjoint solution in $B_4$, while
$q_+$ and $q_-$ are nonnegative normalized adjoint solutions in $B_2$.
Moreover,
\begin{equation}\label{eq:Bauman-q-relations}
q_++q_-=2Mq_0,
\qquad
\frac{w}{a}=\frac{q_+-q_-}{2q_0}.
\end{equation}

Applying the Harnack estimate in
Theorem~\ref{prop:Bauman-normalized} to $q_0$ on
$B_2\subset B_4$, we obtain
\begin{equation}\label{eq:Bauman-q0-comparison}
C^{-1}q_0(0)
\leq q_0\leq
Cq_0(0)
\qquad\text{in }B_2.
\end{equation}
In particular,
\[
\osc_{B_2}q_0\leq Cq_0(0).
\]
Since $q_\pm\geq0$ and $q_++q_-=2Mq_0$,
\[
0\leq q_\pm\leq CMq_0(0)
\qquad\text{in }B_2,
\]
and hence
\[
\osc_{B_2}q_\pm\leq CMq_0(0).
\]

The H\"older estimate in
Theorem~\ref{prop:Bauman-normalized}, applied on
$B_1\subset B_2$, therefore gives
\[
[q_0]_{C^\gamma(B_1)}
\leq Cq_0(0),
\qquad
[q_\pm]_{C^\gamma(B_1)}
\leq CMq_0(0).
\]
Let $N:=q_+-q_-$.
Then
\[
\|N\|_{L^\infty(B_1)}+[N]_{C^\gamma(B_1)}
\leq CMq_0(0),
\]
whereas \eqref{eq:Bauman-q0-comparison} implies
\[
\inf_{B_1}q_0\geq C^{-1}q_0(0).
\]
The elementary quotient estimate
\[
\left[\frac{N}{q_0}\right]_{C^\gamma(B_1)}
\leq
\frac{[N]_{C^\gamma(B_1)}}{\inf_{B_1}q_0}
+
\frac{\|N\|_{L^\infty(B_1)}\cdot
      [q_0]_{C^\gamma(B_1)}}
     {(\inf_{B_1}q_0)^2}
\]
now yields
\[
\left[\frac{N}{q_0}\right]_{C^\gamma(B_1)}
\leq CM.
\]
Using \eqref{eq:Bauman-q-relations}, we conclude that
\[
\left[\frac{w}{a}\right]_{C^\gamma(B_1)}
\leq CM.
\]
The $L^\infty$ estimate follows directly from $|w|\leq Ma$.
Rescaling proves \eqref{eq:Bauman-quotient-estimate}. The final assertion
follows by taking
$w=ap$ and $M=\|p\|_{L^\infty(B_{2r}(z_0))}$.
\end{proof}

\section{Local boundedness of the quotient and proof of Theorem~\ref{thm:main}}\label{sec:bound-main}

Let $Q=\mathcal P(S_1)$. If $E\subset Q$ is measurable, then
\[
\mathrm{d}\xi\mathrm{d}\eta=\varphi_{11}(x)\,\mathrm{d}x,
\qquad
a(\xi,\eta)=\frac1{\varphi_{11}(\mathcal Q(\xi,\eta))}.
\]
Hence
\begin{equation}\label{eq:mass-identity}
\int_E a(\xi,\eta)\,\mathrm{d}\xi\mathrm{d}\eta=|\mathcal Q(E)|.
\end{equation}
This identity will be used below.

\begin{lem}\label{lem:mass-lower}
Let $S_1$ be normalized, set $Q_h:=\mathcal P(S_h)$, and fix
$0<\theta_1<1$. Then there exist $\beta\in(0,1)$ and $c_0,r_0>0$,
depending only on $\lambda,\Lambda$ and $\theta_1$, such that for every
$z\in Q_{\theta_1}$ and every $0<r\leq r_0$ satisfying
$B_r(z)\subset Q_{\theta_1}$,
\[
\int_{B_r(z)}a\,\mathrm{d}\xi\mathrm{d}\eta
\geq c_0r^{2/\beta}.
\]
\end{lem}

\begin{proof}
Set $\theta_2=\frac{1+\theta_1}{2}$. The quantitative compact inclusion of sections
gives
\[
\overline{S_{\theta_1}}\Subset S_{\theta_2}\Subset S_1
\]
with uniform separation. By Lemma~\ref{lem:Caffarelli-C1beta}, there are
$\beta\in(0,1)$ and $C_0<\infty$ such that
\begin{equation}\label{eq:P-holder}
|\mathcal P(x)-\mathcal P(y)|\leq C_0|x-y|^\beta
\qquad\text{for }x,y\in S_{\theta_2}.
\end{equation}

Choose $r_0>0$ so small that
\[
B_{(\frac{r}{2C_0})^{1/\beta}}(x)\subset S_{\theta_2}
\]
for every $x\in\overline{S_{\theta_1}}$ and $0<r\leq r_0$. Let
$z=\mathcal P(x)\in Q_{\theta_1}$. If
$|y-x|\leq(\frac{r}{2C_0})^{1/\beta}$, then
\eqref{eq:P-holder} gives $|\mathcal P(y)-\mathcal P(x)|<r$. Since
$\mathcal P$ is one-to-one,
\[
B_{(\frac{r}{2C_0})^{1/\beta}}(x)\subset \mathcal Q(B_r(z)).
\]
The mass identity \eqref{eq:mass-identity} therefore yields
\[
\int_{B_r(z)}a\,\mathrm{d}\xi\mathrm{d}\eta
=|\mathcal Q(B_r(z))|
\geq c_0r^{2/\beta}.
\]
\end{proof}

\begin{lem}\label{lem:caccioppoli}
Let $w$ satisfy
\[
(gw_\xi)_\xi+w_{\eta\eta}=0
\]
in $B_{2r}(z_0)$, where $\lambda\leq g\leq\Lambda$. Then
\[
\int_{B_r(z_0)}|w_\xi|\,\mathrm{d}\xi\mathrm{d}\eta
\leq Cr\,\osc_{B_{2r}(z_0)}w,
\]
where $C$ depends only on $\lambda$ and $\Lambda$.
\end{lem}

\begin{proof}
By the Caccioppoli inequality,
\[
\int_{B_r(z_0)}|\nabla w|^2\,\mathrm{d}\xi\mathrm{d}\eta
\leq \frac{C}{r^2}
\int_{B_{2r}(z_0)}|w-w_{B_{2r}}|^2\,\mathrm{d}\xi\mathrm{d}\eta
\leq C(\osc_{B_{2r}(z_0)}w)^2.
\]
By the Cauchy--Schwarz inequality,
\[
\int_{B_r(z_0)}|w_\xi|\,\mathrm{d}\xi\mathrm{d}\eta
\leq |B_r|^{1/2}
\left(\int_{B_r(z_0)}|w_\xi|^2\,\mathrm{d}\xi\mathrm{d}\eta\right)^{1/2}
\leq Cr\,\osc_{B_{2r}(z_0)}w.
\]
\end{proof}

\begin{prop}\label{prop:p-bounded}
Let $S_1$ be normalized and let $0<\theta<1$. Then
\[
\|p\|_{L^\infty(\mathcal P(S_\theta))}
\leq C\,\osc_{S_1}u,
\]
where $C$ depends only on $\lambda,\Lambda$ and $\theta$.
\end{prop}

\begin{proof}
Set
\[
\theta_\ast:=\frac{1+\theta}{2},
\qquad
Q_\ast:=\mathcal P(S_{\theta_\ast}).
\]
Let $\beta,c_0,r_0$ be the constants in Lemma~\ref{lem:mass-lower} for
$\theta_\ast$. For $z\in Q_\ast$, set
\[
d(z):=\dist(z,\partial Q_\ast),
\qquad
D(z):=\min\{d(z),r_0\}.
\]
Choose an integer $N$ such that
\[
N>\frac2\beta-1.
\]

Note that $\overline{Q_\ast}\Subset Q$ and $p$ is smooth. While
$D=0$ on $\partial Q_\ast$, the function $|p|D^N$ attains its maximum on
$\overline{Q_\ast}$. Let $y\in Q_\ast$ be a
maximum point. The conclusion is immediate if this maximum is zero, so suppose
$|p(y)|\cdot D(y)^N>0$.

For $z\in B_{\frac{D(y)}{2}}(y)$, the $1$-Lipschitz property of the distance function
gives $D(z)\geq \frac{D(y)}{2}$. By the maximality,
\[
|p(z)|\leq2^N|p(y)|
\qquad\text{in }B_{\frac{D(y)}{2}}(y).
\]
Define
\[
q(z):=\frac{p(z)}{2^N|p(y)|},
\qquad
R_y:=\frac{D(y)}{16}.
\]
Then $|q|\leq1$ in $B_{8R_y}(y)$ and $|q(y)|=2^{-N}$.
If $w_q:=aq$, then
\[
L_g^*a=0,
\qquad
L_g^*w_q=0
\qquad\text{in }B_{8R_y}(y).
\]
Corollary~\ref{cor:Bauman-quotient}, applied with $M=1$, gives
\[
R_y^\gamma\cdot[q]_{C^\gamma(B_{R_y}(y))}\leq C.
\]
Choose $\tau\in(0,\frac{1}{16})$, depending only on $N,\lambda,\Lambda$, so that
$C(16\tau)^\gamma\leq2^{-N-1}$.
If $|z-y|\leq\tau D(y)$, then
\[
|q(z)-q(y)|\leq C\frac{|z-y|^\gamma}{R_y^{\gamma}}\leq C(16\tau)^\gamma\leq2^{-N-1}.
\]
Since $|q(y)|=2^{-N}$, we obtain
\[
|p(z)|\geq\frac12|p(y)|
\qquad\text{in }B_r(y),
\qquad
r:=\tau D(y).
\]

The ball $B_r(y)$ is contained in $Q_\ast$ and $r\leq r_0$, so
Lemma~\ref{lem:mass-lower} and the identity $\widetilde u_\xi=ap$ imply
\[
\int_{B_r(y)}|\widetilde u_\xi|\,\mathrm{d}\xi\mathrm{d}\eta
\geq c|p(y)|r^{2/\beta}.
\]
Since $2r=2\tau D(y)<D(y)$, the ball $B_{2r}(y)$ is compactly contained in
$Q_\ast$. Lemma~\ref{lem:caccioppoli} therefore gives
\[
\int_{B_r(y)}|\widetilde u_\xi|\,\mathrm{d}\xi\mathrm{d}\eta
\leq Cr\,\osc_{B_{2r}(y)}\widetilde u
\leq Cr\,\osc_{S_1}u.
\]
Consequently,
\begin{equation}\label{eq:p-max-bound}
|p(y)|\leq CD(y)^{1-2/\beta}\osc_{S_1}u.
\end{equation}

By Lemma~\ref{lem:P-interior-control}, $\delta:=\inf_{\mathcal P(S_\theta)}D>0$ with a lower bound depending only on $\lambda$, $\Lambda$ and $\theta$. For
$z\in \mathcal P(S_\theta)$, maximality of $y$ and
\eqref{eq:p-max-bound} give
\[
|p(z)|
\leq\delta^{-N}|p(y)|D(y)^N
\leq C\delta^{-N}D(y)^{N+1-2/\beta}\osc_{S_1}u.
\]
Because $N+1-\frac{2}{\beta}>0$ and $D(y)\leq r_0$, the right-hand side is bounded by
$C\osc_{S_1}u$. This proves the proposition.
\end{proof}

\begin{prop}\label{prop:p-holder}
Let $S_1$ be normalized and let $0<\theta<1$. Then
\[
\|p\|_{C^\gamma(\mathcal P(S_\theta))}
\leq C\,\osc_{S_1}u,
\]
where $\gamma\in(0,1)$ depends only on $\lambda$ and $\Lambda$, and $C$
depends only on $\lambda,\Lambda$ and $\theta$.
\end{prop}

\begin{proof}
Set $\theta_1=\frac{1+\theta}{2}$ and $Q_h:=\mathcal P(S_h)$.
Lemma~\ref{lem:P-interior-control} gives a quantitative separation between
$Q_\theta$ and $\partial Q_{\theta_1}$. Proposition~\ref{prop:p-bounded},
applied at the level $\theta_1$, gives
\[
\|p\|_{L^\infty(Q_{\theta_1})}
\leq C\osc_{S_1}u.
\]

By Lemma~\ref{lem:P-interior-control}, there exists $r_\ast>0$, depending
only on $\lambda,\Lambda$ and $\theta$, such that
\[
B_{8r_\ast}(z)\subset Q_{\theta_1}
\qquad
\text{for every }z\in Q_\theta.
\]
For each $z\in Q_\theta$, apply Corollary~\ref{cor:Bauman-quotient} with
$w=ap$ in $B_{8r_\ast}(z)$. Since
\[
\|p\|_{L^\infty(B_{2r_\ast}(z))}
\leq
\|p\|_{L^\infty(Q_{\theta_1})}
\leq C\osc_{S_1}u,
\]
we obtain
\[
\|p\|_{L^\infty(B_{r_\ast}(z))}
+
r_\ast^\gamma[p]_{C^\gamma(B_{r_\ast}(z))}
\leq
C\osc_{S_1}u.
\]

We now pass from the local estimate to $Q_\theta$. If
$z_1,z_2\in Q_\theta$ and $|z_1-z_2|<r_\ast$, then the preceding estimate
applied at $z_1$ gives
\[
|p(z_1)-p(z_2)|
\leq
C\osc_{S_1}u\, |z_1-z_2|^\gamma.
\]
If $|z_1-z_2|\geq r_\ast$, then
\[
|p(z_1)-p(z_2)|
\leq
2\|p\|_{L^\infty(Q_{\theta_1})}
\leq
C\osc_{S_1}u
\leq
C r_\ast^{-\gamma}\osc_{S_1}u\, |z_1-z_2|^\gamma.
\]
Since $r_\ast$ is quantitatively controlled, this gives
\[
\|p\|_{C^\gamma(Q_\theta)}
\leq
C\osc_{S_1}u.
\]
\end{proof}

\begin{proof}[Proof of Theorem \ref{thm:main}]
We first estimate $u_1$. Proposition~\ref{prop:p-holder} gives
\[
\|p\|_{C^\gamma(\mathcal P(S_\theta))}\leq C\osc_{S_1}u.
\]
By Lemma~\ref{lem:Caffarelli-C1beta}, the map
$\mathcal P(x)=(\varphi_1(x),x_2)$ is $C^\beta$ on $S_\theta$, with
\[
|\mathcal P(x)-\mathcal P(y)|\leq C|x-y|^\beta
\qquad\text{for }x,y\in S_\theta.
\]
Since $u_1=p\circ\mathcal P$, we obtain
\begin{equation}\label{eq:u1-holder-final}
\|u_1\|_{C^{\beta\gamma}(S_\theta)}\leq C\osc_{S_1}u.
\end{equation}
The constants depend only on $\lambda,\Lambda$ and $\theta$.

We apply the same argument in the second variable. Write
$(\xi,\eta)=(x_1,\varphi_2(x_1,x_2))$,
let $\widehat{\mathcal Q} $ be the inverse map, and set
\[
\widehat a=(\varphi_{22}\circ\widehat{\mathcal Q})^{-1},
\qquad
\widehat p=u_2\circ\widehat{\mathcal Q},
\qquad
\widehat g=\det D^2\varphi\circ\widehat{\mathcal Q}.
\]
The transformed operator is
$\partial_{\xi\xi}+\widehat g\partial_{\eta\eta}$, and
$\widehat a$ and $\widehat a\widehat p$ solve its adjoint equation. The mass
identity and all preceding estimates are unchanged after interchanging the two
coordinates. Hence
\begin{equation}\label{eq:u2-holder-final}
\|u_2\|_{C^{\beta\gamma}(S_\theta)}\leq C\osc_{S_1}u.
\end{equation}
Combining \eqref{eq:u1-holder-final} and \eqref{eq:u2-holder-final} proves
\eqref{eq:main-estimate-intro} with $\alpha=\beta\gamma$.
\end{proof}
\section{Proof of Theorem~\ref{thm:Liouville}}\label{sec:liouville}

\begin{proof}[Proof of Theorem~\ref{thm:Liouville}]
After subtracting the tangent plane of $\varphi$ at the origin, we may assume
\[
\varphi(0)=0,
\qquad
\nabla\varphi(0)=0,
\]
and write
\[
S_h:=S_\varphi(0,h)=\{x\in\R^2:\varphi(x)<h\}.
\]
Every $S_h$ is bounded. Indeed, suppose that $\overline{S_h}$ were unbounded.
Choose $x_k\in\overline{S_h}$ with $|x_k|\to\infty$ and, after passing to a
subsequence, assume $\frac{x_k}{|x_k|}\to v\in\mathbb S^1$. Since
$0\in\overline{S_h}$ and $\overline{S_h}$ is convex, for every fixed $t>0$ one
has $\frac{t}{|x_k|}x_k\in\overline{S_h}$
for all sufficiently large $k$. Passing to the limit gives
$tv\in\overline{S_h}$. Thus $t\mapsto\varphi(tv)$ is convex, nonnegative,
has right derivative zero at $t=0$, and is bounded above by $h$ on
$[0,\infty)$. Its derivative must therefore vanish identically, contradicting
strict convexity. The sections are increasing, and they exhaust $\R^2$ because
$\varphi$ is finite everywhere.

For each $h>0$, let $\mathcal A_hx=A_hx+c_h$
be an orientation-preserving affine normalization of $S_h$. By the standard
normalization and volume estimates for Monge--Amp\`ere sections, see for
example \cite{GH,GuBook,LeBook}, we may arrange that, with
$\widetilde x_h:=\mathcal A_h(0)$,
\begin{equation}\label{eq:large-section-normalization}
B_{c_0}(\widetilde x_h)
\subset \mathcal A_hS_h
\subset B_{C_0}(\widetilde x_h),
\qquad
\det A_h=h^{-1},
\end{equation}
where $0<c_0<C_0<\infty$ depend only on $\lambda$ and $\Lambda$.
Starting from a standard normalization, the determinant can be fixed exactly
by multiplying the linear part by a scalar bounded above and below by
structural constants, since $|S_h|\simeq h$ in dimension two.
We claim that
\begin{equation}\label{eq:Ah-goes-zero}
\|A_h\|\to0
\qquad\text{as }h\to\infty.
\end{equation}
Fix $R>0$. Since $S_h$ exhausts $\R^2$, one has $B_R\subset S_h$ for all
sufficiently large $h$. For every unit vector $v$, both $0$ and $Rv$ then
belong to $S_h$, and hence
\[
R|A_hv|
=
|\mathcal A_h(Rv)-\mathcal A_h(0)|
\leq
\diam(\mathcal A_hS_h)
\leq2C_0.
\]
Thus
$\|A_h\|\leq\frac{2C_0}{R}$ eventually. Since $R$ is arbitrary,
\eqref{eq:Ah-goes-zero} follows.

Define
\[
\widetilde\varphi_h(y):=\frac1h\varphi(\mathcal A_h^{-1}y),
\qquad
\widetilde u_h(y):=u(\mathcal A_h^{-1}y).
\]
Then
\[
\widetilde\varphi_h(\widetilde x_h)=0,
\qquad
\nabla\widetilde\varphi_h(\widetilde x_h)=0,
\]
and
\[
S_{\widetilde\varphi_h}(\widetilde x_h,1)=\mathcal A_hS_h,
\qquad
S_{\widetilde\varphi_h}(\widetilde x_h,1/2)=\mathcal A_hS_{h/2}.
\]
Moreover, because $\det A_h=h^{-1}$,
\[
\det D_y^2\widetilde\varphi_h(y)
=
\det D_x^2\varphi(\mathcal A_h^{-1}y),
\]
and affine covariance gives
\[
L_{\widetilde\varphi_h}\widetilde u_h=0
\qquad\text{in }\mathcal A_hS_h.
\]
After translating the $y$-variables by $-\widetilde x_h$, the section in
\eqref{eq:large-section-normalization} is normalized. Fix the exponent
$\alpha_0\in(0,1)$ supplied by Theorem~\ref{thm:main} at the relative scale
$1/2$. Applying that theorem to $\widetilde u_h$ yields
\begin{equation}\label{eq:affine-gradient-liouville}
\left|
A_h^{-T}\bigl(\nabla u(x)-\nabla u(x')\bigr)
\right|
\leq
C\,\osc_{S_h}u\cdot\,|A_h(x-x')|^{\alpha_0}
\end{equation}
for every $x,x'\in S_{h/2}$, with $C$ independent of $h$.
Here $A_h^{-T}$ denotes the transpose of $A_h^{-1}$.

For each $h$, choose a unit right singular vector $e_h$ of $A_h$ associated
with its smallest singular value. Thus
\begin{equation}\label{eq:min-singular-direction}
|A_he_h|=\sigma_{\min}(A_h)=\|A_h^{-1}\|^{-1}.
\end{equation}
Using
\[
\left(\nabla u(x)-\nabla u(x')\bigr)\cdot e_h=\bigl(A_h^{-T}\bigl(\nabla u(x)-\nabla u(x')\bigr)\right)\cdot(A_he_h),
\]
together with \eqref{eq:affine-gradient-liouville} and
\eqref{eq:min-singular-direction}, we obtain
\begin{equation}\label{eq:directional-gradient-liouville}
\left|\bigl(\nabla u(x)-\nabla u(x')\bigr)\cdot e_h\right|
\leq
C\frac{\osc_{S_h}u}{\|A_h^{-1}\|}\cdot
\|A_h\|^{\alpha_0}\cdot|x-x'|^{\alpha_0}.
\end{equation}
The translation part of $\mathcal A_h$ cancels when distances are taken, and
\eqref{eq:large-section-normalization} implies
\[
\diam S_h
\leq C\|A_h^{-1}\|.
\]
Since $0\in S_h$, the linear growth assumption therefore yields
\begin{equation*}
\osc_{S_h}u
\leq C\bigl(1+\|A_h^{-1}\|\bigr).
\end{equation*}
By \eqref{eq:Ah-goes-zero}, we have
$\|A_h^{-1}\|\to\infty$, and 
\begin{equation}\label{eq:osc-over-long-axis}
\frac{\osc_{S_h}u}{\|A_h^{-1}\|}\leq C
\end{equation}
for all sufficiently large $h$.

Choose a sequence $h_k\to\infty$. After passing to a subsequence, compactness
of the unit circle gives $e_{h_k}\to e_\infty\in\mathbb S^1$.
Fix arbitrary $x,x'\in\R^2$. For all sufficiently large $k$, one has
$x,x'\in S_{\frac{h_k}{2}}$. Combining
\eqref{eq:directional-gradient-liouville},
\eqref{eq:osc-over-long-axis}, and \eqref{eq:Ah-goes-zero}, and then letting
$k\to\infty$, gives
\[
\bigl(\nabla u(x)-\nabla u(x')\bigr)\cdot e_\infty=0.
\]
Hence $e_\infty\cdot\nabla u$ is constant in $\R^2$. Therefore
\begin{equation}\label{eq:hessian-kernel-direction}
D^2u\,e_\infty=0
\qquad\text{in }\R^2.
\end{equation}
Let $e_\infty^\perp$ be a unit vector orthogonal to $e_\infty$. Since $D^2u$
is a symmetric $2\times2$ matrix, \eqref{eq:hessian-kernel-direction} implies
that, pointwise,
\[
D^2u=\mu\,e_\infty^\perp\otimes e_\infty^\perp
\]
for a scalar function $\mu$. Substitution into the equation yields
\[
0=L_\varphi u
=\tr(\Phi D^2u)
=\mu\,(e_\infty^\perp)^T\Phi e_\infty^\perp.
\]
The cofactor matrix $\Phi$ is positive definite, so
$(e_\infty^\perp)^T\Phi e_\infty^\perp>0$. Consequently $\mu\equiv0$,
hence $D^2u\equiv0$ and $u$ is affine.
\end{proof}


\begin{thebibliography}{9999}

\bibitem[Bau84]{Ba}
P.~Bauman,
Positive solutions of elliptic equations in nondivergence form and their adjoints,
\emph{Ark. Mat.} \textbf{22} (1984), no.~2, 153--173,
\url{https://doi.org/10.1007/BF02384378}.

\bibitem[Caf90]{CaLoc}
L.~A. Caffarelli,
A localization property of viscosity solutions to the Monge--Amp\`ere equation and their strict convexity,
\emph{Ann.\ of Math. (2)} \textbf{131} (1990), no.~1, 129--134,
\url{https://doi.org/10.2307/1971509}.

\bibitem[Caf91]{Ca}
L.~A. Caffarelli,
Some regularity properties of solutions of Monge--Amp\`ere equation,
\emph{Comm.\ Pure Appl.\ Math.} \textbf{44} (1991), no.~8--9, 965--969,
\url{https://doi.org/10.1002/cpa.3160440809}.

\bibitem[CG96]{CaGuRA}
L.~A. Caffarelli and C.~E. Guti\'errez,
Real analysis related to the Monge--Amp\`ere equation,
\emph{Trans.\ Amer.\ Math.\ Soc.} \textbf{348} (1996), no.~3, 1075--1092,
\url{https://doi.org/10.1090/S0002-9947-96-01473-0}.

\bibitem[CG97]{CG}
L.~A. Caffarelli and C.~E. Guti\'errez,
Properties of the solutions of the linearized Monge--Amp\`ere equation,
\emph{Amer.\ J. Math.} \textbf{119} (1997), no.~2, 423--465,
\url{https://doi.org/10.1353/ajm.1997.0010}.

\bibitem[Cui26]{Cui26}
G.~Cui,
Gradient potential estimates for linearized Monge--Amp\`ere equations,
\emph{arXiv preprint} arXiv:2606.28910 (2026),
\url{https://doi.org/10.48550/arXiv.2606.28910}.

\bibitem[Fig17]{FiBook}
A.~Figalli,
\emph{The Monge--Amp\`ere equation and its applications},
Zurich Lectures in Advanced Mathematics,
European Mathematical Society, Z\"urich, 2017,
\url{https://doi.org/10.4171/170}.


\bibitem[Gut16]{GuBook}
C.~E. Guti\'errez,
\emph{The Monge--Amp\`ere equation},
2nd ed.,
Progress in Nonlinear Differential Equations and Their Applications, vol.~89,
Birkh\"auser, Cham, 2016,
\url{https://doi.org/10.1007/978-3-319-43374-5}.

\bibitem[GH00]{GH}
C.~E. Guti\'errez and Q.~Huang,
Geometric properties of the sections of solutions to the Monge--Amp\`ere equation,
\emph{Trans.\ Amer.\ Math.\ Soc.} \textbf{352} (2000), no.~9, 4381--4396,
\url{https://doi.org/10.1090/S0002-9947-00-02491-0}.

\bibitem[GN11]{GN1}
C.~E. Guti\'errez and T.~Nguyen,
Interior gradient estimates for solutions to the linearized Monge--Amp\`ere equation,
\emph{Adv.\ Math.} \textbf{228} (2011), no.~4, 2034--2070,
\url{https://doi.org/10.1016/j.aim.2011.06.035}.

\bibitem[GN15]{GN2}
C.~E. Guti\'errez and T.~Nguyen,
Interior second derivative estimates for solutions to the linearized Monge--Amp\`ere equation,
\emph{Trans.\ Amer.\ Math.\ Soc.} \textbf{367} (2015), no.~7, 4537--4568,
\url{https://doi.org/10.1090/S0002-9947-2015-06048-6}.

\bibitem[GT06]{GuTo}
C.~E. Guti\'errez and F.~Tournier,
$W^{2,p}$-estimates for the linearized Monge--Amp\`ere equation,
\emph{Trans.\ Amer.\ Math.\ Soc.} \textbf{358} (2006), no.~11, 4843--4872,
\url{https://doi.org/10.1090/S0002-9947-06-04189-4}.

\bibitem[Hei59]{He}
E.~Heinz,
On elliptic Monge--Amp\`ere equations and Weyl's embedding problem,
\emph{J. Analyse Math.} \textbf{7} (1959), 1--52,
\url{https://doi.org/10.1007/BF02787679}.

\bibitem[KLWZ26]{KLWZ}
Y.~H. Kim, N.~Q. Le, L.~Wang, and B.~Zhou,
Singular Abreu equations and linearized Monge--Amp\`ere equations with drifts,
\emph{J. Eur.\ Math.\ Soc.} \textbf{28} (2026), no.~9, 4105--4148,
\url{https://doi.org/10.4171/JEMS/1548}.

\bibitem[Le18]{LeSem}
N.~Q. Le,
H\"older regularity of the 2D dual semigeostrophic equations via analysis of linearized Monge--Amp\`ere equations,
\emph{Comm.\ Math.\ Phys.} \textbf{360} (2018), no.~1, 271--305,
\url{https://doi.org/10.1007/s00220-018-3125-9}.

\bibitem[Le24]{LeBook}
N.~Q. Le,
\emph{Analysis of Monge--Amp\`ere equations},
Graduate Studies in Mathematics, vol.~240,
American Mathematical Society, Providence, RI, 2024,
\url{https://doi.org/10.1090/gsm/240}.

\bibitem[LN14]{LN1}
N.~Q. Le and T.~Nguyen,
Global $W^{2,p}$ estimates for solutions to the linearized Monge--Amp\`ere equations,
\emph{Math.\ Ann.} \textbf{358} (2014), no.~3--4, 629--700,
\url{https://doi.org/10.1007/s00208-013-0974-6}.

\bibitem[LN17]{LN2}
N.~Q. Le and T.~Nguyen,
Global $W^{1,p}$ estimates for solutions to the linearized Monge--Amp\`ere equations,
\emph{J. Geom.\ Anal.} \textbf{27} (2017), no.~3, 1751--1788,
\url{https://doi.org/10.1007/s12220-016-9739-2}.

\bibitem[LS13]{LS}
N.~Q. Le and O.~Savin,
Boundary regularity for solutions to the linearized Monge--Amp\`ere equations,
\emph{Arch.\ Ration.\ Mech.\ Anal.} \textbf{210} (2013), no.~3, 813--836,
\url{https://doi.org/10.1007/s00205-013-0653-5}.

\bibitem[Liu21]{Liu}
J.~Liu,
Interior $C^2$ estimate for Monge--Amp\`ere equation in dimension two,
\emph{Proc.\ Amer.\ Math.\ Soc.} \textbf{149} (2021), no.~6, 2479--2486,
\url{https://doi.org/10.1090/proc/15459}.

\bibitem[Mor38]{Mor}
C.~B. Morrey, Jr.,
On the solutions of quasi-linear elliptic partial differential equations,
\emph{Trans.\ Amer.\ Math.\ Soc.} \textbf{43} (1938), no.~1, 126--166,
\url{https://doi.org/10.1090/S0002-9947-1938-1501936-8}.

\bibitem[Nir53]{Ni}
L.~Nirenberg,
On nonlinear elliptic partial differential equations and H\"older continuity,
\emph{Comm.\ Pure Appl.\ Math.} \textbf{6} (1953), no.~1, 103--156,
\url{https://doi.org/10.1002/cpa.3160060105}.

\bibitem[Saf88]{Saf}
M.~V. Safonov,
Unimprovability of estimates of H\"older constants for solutions of linear
elliptic equations with measurable coefficients,
\emph{Math. USSR-Sb.} \textbf{60} (1988), no.~1, 269--281,
\url{https://doi.org/10.1070/SM1988v060n01ABEH003167}.

\bibitem[Sav10]{Sa}
O.~Savin,
A Liouville theorem for solutions to the linearized Monge--Amp\`ere equation,
\emph{Discrete Contin.\ Dyn.\ Syst.} \textbf{28} (2010), no.~3, 865--873,
\url{https://doi.org/10.3934/dcds.2010.28.865}.

\bibitem[TZ20]{TZ}
L.~Tang and Q.~Zhang,
Interior $C^{1,\alpha}$ regularity for the linearized Monge--Amp\`ere equation with VMO type coefficients,
\emph{Adv.\ Oper.\ Theory} \textbf{5} (2020), no.~1, 204--218,
\url{https://doi.org/10.1007/s43036-019-00012-1}.

\bibitem[TW08a]{TW}
N.~S. Trudinger and X.-J. Wang,
The Monge--Amp\`ere equation and its geometric applications,
in \emph{Handbook of Geometric Analysis}, No.~1,
Advanced Lectures in Mathematics (ALM), vol.~7,
International Press, Somerville, MA, 2008, pp.~467--524.

\bibitem[TW08b]{TrWa}
N.~S. Trudinger and X.-J. Wang,
Boundary regularity for the Monge--Amp\`ere and affine maximal surface equations,
\emph{Ann.\ of Math. (2)} \textbf{167} (2008), no.~3, 993--1028,
\url{https://doi.org/10.4007/annals.2008.167.993}.

\bibitem[Wan25]{Wa}
L.~Wang,
Interior H\"older regularity of the linearized Monge--Amp\`ere equation,
\emph{Calc.\ Var.\ Partial Differential Equations} \textbf{64} (2025),
no.~1, Paper No.~17,
\url{https://doi.org/10.1007/s00526-024-02885-4}.

\end{thebibliography}
\end{document}